\documentclass[11pt]{amsart}
\usepackage{amscd,amsfonts,amsmath,amsthm,amssymb,mathrsfs,verbatim}
\usepackage{pstcol,pst-plot,pst-3d}
\usepackage{lmodern,pst-node, xfrac}
\usepackage{amsmath}
\usepackage[table]{xcolor}

\usepackage{pstricks}

\newcommand\z{\cellcolor{darkgray!10}}

\def\NZQ{\mathbb}               
\def\NN{{\NZQ N}}

\def\ZZ{{\NZQ Z}}

\def\frk{\frak}               

\def\Phi{{\frk n}}
\def\Phi{{\frk N}}
\def\MA{{\mathcal A}}
\def\MB{{\mathcal B}}

\def\MM{{\mathcal M}}

\def\MS{{\mathcal S}}

\def\MG{{\mathcal G}}
\def\ML{{\mathcal L}}

\def\opn#1#2{\def#1{\operatorname{#2}}} 
\opn\chara{char} \opn\length{\ell} \opn\pd{pd} \opn\rk{rk}
\opn\projdim{proj\,dim} \opn\injdim{inj\,dim} \opn\rank{rank}
\opn\depth{depth} \opn\grade{grade} \opn\height{height}
\opn\embdim{emb\,dim} \opn\codim{codim}

\opn\Tr{Tr} \opn\bigrank{big\,rank}
\opn\superheight{superheight}\opn\lcm{lcm}
\opn\trdeg{tr\,deg}
\opn\reg{reg} \opn\lreg{lreg} \opn\ini{in} \opn\lpd{lpd}
\opn\size{size}\opn\bigsize{bigsize}
\opn\cosize{cosize}\opn\bigcosize{bigcosize}
\opn\sdepth{sdepth}\opn\sreg{sreg}
\opn\link{link}\opn\fdepth{fdepth}
\opn\div{div} \opn\Div{Div} \opn\cl{cl} \opn\Cl{Cl}
\opn\Spec{Spec} \opn\Supp{Supp} \opn\supp{supp} \opn\Sing{Sing}
\opn\Ass{Ass} \opn\Min{Min}\opn\Mon{Mon} \opn\dstab{dstab} \opn\astab{astab}
\opn\Ann{Ann} \opn\Rad{Rad} \opn\Soc{Soc}
\opn\Im{Im} \opn\Ker{Ker} \opn\Coker{Coker} \opn\Am{Am}
\opn\Hom{Hom} \opn\Tor{Tor} \opn\Ext{Ext} \opn\End{End}
\opn\Aut{Aut} \opn\id{id} \opn\span{span}

\opn\nat{nat}
\opn\pff{pf}
\opn\Pf{Pf} \opn\GL{GL} \opn\SL{SL} \opn\mod{mod} \opn\ord{ord}
\opn\Gin{Gin} \opn\Hilb{Hilb}\opn\sort{sort}
\opn\aff{aff} \opn\con{conv} \opn\relint{relint} \opn\st{st}
\opn\lk{lk} \opn\cn{cn} \opn\core{core} \opn\vol{vol}
\opn\link{link} \opn\star{star}\opn\lex{lex} \opn\Gr{Gr}
\opn\gr{gr}

\def\pot#1#2{#1[\kern-0.28ex[#2]\kern-0.28ex]}

\opn\dirlim{\underrightarrow{\lim}}
\opn\inivlim{\underleftarrow{\lim}}
\def\Implies{\ifmmode\Longrightarrow \else
        \unskip${}\Longrightarrow{}$\ignorespaces\fi}
\def\implies{\ifmmode\Rightarrow \else
        \unskip${}\Rightarrow{}$\ignorespaces\fi}
\def\iff{\ifmmode\Longleftrightarrow \else
        \unskip${}\Longleftrightarrow{}$\ignorespaces\fi}

\let\:=\colon
\newtheorem{Theorem}{Theorem}[section]
\newtheorem{Lemma}[Theorem]{Lemma}
\newtheorem{Corollary}[Theorem]{Corollary}
\newtheorem{Proposition}[Theorem]{Proposition}
\newtheorem{Remark}[Theorem]{Remark}

\let\epsilon\varepsilon
\let\kappa=\varkappa
\def\qed{\ifhmode\textqed\fi
      \ifmmode\ifinner\quad\qedsymbol\else\dispqed\fi\fi}
\def\textqed{\unskip\nobreak\penalty50
       \hskip2em\hbox{}\nobreak\hfil\qedsymbol
       \parfillskip=0pt \finalhyphendemerits=0}
\def\dispqed{\rlap{\qquad\qedsymbol}}

\opn\dis{dis}
\def\pnt{{\raise0.5mm\hbox{\large\bf.}}}

\opn\Lex{Lex}

\begin{document}

\title{Unboundedness of Markov complexity of monomial curves in $\mathbb{A}^n$ for $n\geq 4$.}



\author{Dimitra Kosta         \and
        Apostolos Thoma 
}


\address{Dimitra Kosta, School of Mathematics and Statistics, University of Glasgow, United Kingdom }
\email{Dimitra.Kosta@glasgow.ac.uk}

\address{Apostolos Thoma, Department of Mathematics, University of Ioannina, Ioannina 45110, Greece }
\email{athoma@uoi.gr}

\begin{abstract}
Computing the complexity of Markov bases is an extremely challenging problem; no formula is known in general and there are very few classes of toric ideals for which the Markov complexity has been computed. A monomial curve $C$  in $\mathbb{A}^3$ has Markov complexity $m(C)$ two or three. Two if the monomial curve is complete intersection and three otherwise. Our main result shows that
there is no $d\in \NN$ such that $m(C)\leq d$ for all monomial curves $C$ in $\mathbb{A}^4$.
The same result is true even if we restrict to complete intersections. We extend this result to all monomial curves in $\mathbb{A}^n, n\geq 4$.


\end{abstract}

\maketitle

\section{Introduction}
\label{intro}

 Much of the current interest in Markov bases of toric ideals and their complexity, at least from an applications perspective, began with the seminal paper \cite{DS}, which constitutes one of the first connections between commutative algebra and statistics. This work proposes algebraic algorithms to construct a connected Markov chain over high-dimensional contingency tables with fixed marginals, using Gr\"obner bases. Motivated by this work, the Markov bases of certain contigency tables were studied in \cite{AT} and also the first examples of matrices with finite Markov complexity were provided.

 In an effort to better understand the Markov basis $\MM(\MA)$ of a toric ideal associated to a matrix $\MA$, the study of auxiliary generating sets, such as the indispensable set $\MS(\MA)$ and the Graver basis $\MG(\MA)$ of $\MA$, is employed. Building on the work by  \cite{AT}, it was proven in \cite{SS} that the Markov complexity is bounded above by the Graver complexity, and since the latter one is finite, the Markov complexity is also finite.

 In \cite{CKT} a geometric description is given for the elements of the Markov basis $\MM(\MA)$ and the indispensable set $\MS(\MA)$, which uses the correspondence between fibers of $\MA$ and certain connected components of a certain simplicial complex associated to $\MA$. At the same time, in a more algebraic approach adopted to describe the indispensable set $\MS(\MA)$, the notion of proper semiconformal decomposition was introduced in \cite{HS}. Building on this idea, a complete algebraic characterization for the elements of the indispensable set $\MS(\MA)$ and the Markov basis $\MM(\MA)$ is provided in \cite{HaraThomaVladoiu14} using extended notions of conformality, i.e. conformal, semiconformal, strongly semiconformal (see Section~\ref{section:preliminaries} for definitions). This description will be employed throughout this paper.

 Moreover, Graver bases and their complexity have also very important applications in Integer Programming, where considerable effort has been put into estimating the growth of the Graver complexity, as this specifies the time complexity of various $n$-fold integer programmes (see \cite[Chapter 4]{Onn}). Most efforts in the Integer Programming community have focused on proving exponential lower bounds for the Graver complexity of complete bipartite graphs, as in \cite{BerOnn}, \cite{KudTak} and \cite{FinHem}. It is still an open conjecture that the Graver complexity of the complete bipartite graph $K_{3,m}$ is equal to $3^{m-1}$.

 In \cite{HaraThomaVladoiu14}, it is shown that the Markov complexity of the monomial curve $\MA = (n_1, n_2, n_3)$ is equal to
 two if the toric ideal $I_{\MA}$ is complete intersection and equal to three otherwise, answering a question posed by Santos and Sturmfels (see \cite[Example 6]{SS}). However, computing the complexity $m(\MA)$ of Markov bases is an extremely challenging problem; no formula for the $m(A)$ is known in general and there are very few classes of toric ideals in the literature for which the complexity has been computed \cite{AT, HS, SS, HaraThomaVladoiu14}.

 The purpose of this paper is to study the Markov complexity $m(\MA)$ of monomial curves in  $\mathbb{A}^m, m\geq 4$ and demonstrate that the result of \cite{HaraThomaVladoiu14}, which bounds the Markov complexity of complete intersection monomial curves in $\mathbb{A}^3$ by their codimension, is a special property of monomial curves in $\mathbb{A}^3$ and cannot be generalised to higher dimensions. In particular, we obtain that complete intersection monomial curves in $\mathbb{A}^4$ may have arbitrary large Markov complexity; this is a corollary of the following Theorem which is the main result of this paper.

\medskip
\textbf{Theorem~\ref{large}}. \emph{Monomial curves in $\mathbb{A}^4$ may have arbitrary large Markov complexity.}
\medskip

To prove this, we need to find a family  $\MA_{r} = (a_1(r), a_2(r), a_3(r), a_4(r))$ of monomial curves in $\mathbb{A}^4$, where the numbers $a_1(r), a_2(r), a_3(r), a_4(r)$ depend on a parameter $r$, such that the Markov complexity of $\MA_n$,
the $n^{th}$ member of the family, is at least $n$. That meant finding an element of type $n$ that belongs
to $\MM(\MA_n^{(n)})$. After several months working with the computational commutative algebra package 4ti2 \cite{4ti2},
we did find one such family, $\MA_n=(1,n, n^2-n, n^2-1)$ and an element of type $n$ in $\ML(\MA_n^{(n)})$, which we managed
to prove in a simple way belongs to every Markov basis of $\MA_n^{(n)}$.

The paper is organised in the following manner. Section~\ref{section:preliminaries} contains all the necessary
definitions and properties of different types of decompositions. It also features Theorem~\ref{restriction} which states
that Markov bases of higher Lawrence liftings behave well with respect to elimination and implies necessary conditions
for the Markov complexity to be equal to 2. In Section~\ref{Section: MonomialCurve}, we provide the guiding example of a family
of monomial curves in $\mathbb{A}^4$ with arbitrary large Markov complexity.
Then, the final Section~\ref{Section: MainResult} includes the proof of our main result Theorem~\ref{large} which we also generalise
to monomial curves in $\mathbb{A}^m, m\geq 4$.

\section{Preliminaries}

\label{section:preliminaries}


Consider a set of vectors $\MA=\{{\bf a}_1,\ldots,{\bf a}_n\}\subset\NN^m$ and the corresponding matrix
$A\in\MM_{m\times n}(\NN)$ whose columns are the vectors of $\MA$, where  $n,m\in \NN$. We  let $\ML({\MA}):=\Ker_{\ZZ}(A)$ be the corresponding sublattice of $\ZZ^n$  and denote by $I_{\MA}$ the corresponding  toric ideal of $\MA$ in $\mathbb{k}[x_1,\ldots, x_n]$, where $\mathbb{k}$ is a field. We recall that $I_{\MA}$ is generated by all binomials of the form $x^{\bf u }-x^{\bf w}$ where
$\ {\bf u}-{\bf w} \in  \mathcal{L}(\MA)$.

A \emph{Markov basis} $\MM $ of $\MA $ is a finite subset of $\Ker_{\ZZ}(A)$ such that whenever ${\bf w}, {\bf u} \in \mathbb{N}^n$ with $x^{{\bf w}}, x^{{\bf u}}$ in the same fiber (namely ${\bf w}-{\bf u} \in \Ker_{\ZZ}(A)$ ), there exists a subset $ \{ {\bf v}_i : i = 1, \cdots, s \}$ of $\MM$ that connects ${\bf w}$ to ${\bf u}$. This means that $({\bf w}- {{\sum^p_{i=1}}}{{\bf v}_i})\in \NN^n$  for all $1\leq p\leq s$ and   ${\bf w}- {\bf u}=\sum^s_{i=1}{\bf v}_i $. We call a Markov basis $\MM$ of $\MA$ {\it minimal} if no subset of $\MM$ is a Markov basis of $\MA$. For a vector ${\bf u}\in \mathcal{L}(\MA)$, we denote by $\bf u^+$, ${\bf u}^-$ the unique vectors in $ \NN^{n}$ such that  $\bf u= \bf u^+ -\bf u^-$. According to a classical result by Diaconis and Sturmfels, if $\MM$  is a minimal Markov basis of $\MA$, then the set $\{x^{\bf u^+}-x^{\bf u^-}: \ {\bf u}\in \MM\}$ is a minimal generating set of $I_{\MA}$ (see \cite[Theorem 3.1]{DS}). The union of all minimal Markov bases of $\MA$, where we identify  elements that differ by a sign, is called the {\em universal Markov basis} of $\MA$ and is denoted by ${\MM}(\MA)$ (see \cite[Definition 3.1]{HS}).

The {\it indispensable subset} of the universal Markov basis $\MM(\MA)$, which is denoted by $\MS(\MA)$, is the intersection of all minimal Markov bases of $\MA$ via the same identification.  The {\it Graver basis} of $\MA$,  $\MG(\MA)$,  is the subset of $\mathcal{L}(\MA)$ whose elements
have {  no} {\it  proper conformal decomposition}; namely, an element ${\bf u}\in \mathcal{L}(\MA)$
belongs to the Graver basis $\MG(\MA)$ if whenever ${\bf u}$ can be written in the form ${\bf v}+_{c} {\bf w}$,
where ${\bf v}, {\bf w} \in \mathcal{L}(\MA)$ and  ${\bf u}^+={\bf v}^++ {\bf w}^+$,  ${\bf u}^-={\bf v}^-+ {\bf w}^-$,
we conclude that either ${\bf v}={\bf 0}$ or  ${\bf w}={\bf 0}$, see \cite[Section 4]{St}. The {\it Graver basis} of $\MA$ is always a finite set and contains the universal Markov basis of $\MA$,  see \cite[Section  7]{St}. Therefore, we have the following inclusions
 \[\MS(\MA)\subseteq\MM(\MA)\subseteq\MG(\MA) .\]

 The notion of a semiconformal decomposition was introduced in \cite[Definition 3.9]{HS}.
 Let ${\bf u},{\bf v},{\bf w}\in\ML({\MA})$.
 We say that  ${\bf u}={\bf v}+_{sc}{\bf w}$  is a {\bf semiconformal decomposition}
 of ${\bf u}$ if ${\bf u}={\bf v}+{\bf w}$ and ${\bf v}(i)>0$ implies that ${\bf w}(i)\geq 0$ and ${\bf w}(i)<0$ implies
 that ${\bf v}(i)\leq 0$ for  $1\leq i\leq n$.
 Here ${\bf v}(i)$ denotes the $i^{th}$ coordinate of the vector ${\bf v}$. We call the decomposition {\bf proper} if both ${\bf v}, {\bf w}$ are nonzero.
 It is easy to see that  ${\bf u}={\bf v}+_{sc} {\bf w}$  if and only if  ${\bf u}^+\geq {\bf v}^+$ and ${\bf u}^-\geq {\bf w}^-$.
 We remark that  $\bf 0$ cannot be written as the semiconformal sum of two nonzero vectors  since $\ML({\MA})\cap \NN^n=\{\bf 0\}$.

  The lack of a proper semiconformal decomposition is not only a sufficient condition for an element to be in ${\MS}(\MA)$ as was shown in  \cite[Lemma 3.10]{HS}, but it is also a necessary condition by \cite[Proposition 1.1]{HaraThomaVladoiu14}.
 \begin{Proposition}
\label{indispensable}
 The set of indispensable elements ${\MS}(\MA)$ of $\MA$ consists  of all nonzero vectors in $\ML({\MA})$ with no proper   semiconformal decomposition.
\end{Proposition}

Let ${\bf u}, {\bf u}_1,\ldots,{\bf u}_l \in\ML({\MA})$, $l\geq 2$. We say that ${\bf u}=_{ssc}{\bf u}_1+\cdots+{\bf u}_l$,  is a {\bf strongly semiconformal decomposition} if  ${\bf u}={\bf u}_1+\cdots+{\bf u}_l$ and the following conditions are satisfied:
\[
{\bf u}^+>{\bf u}_1^+ \quad \text{ and } \quad {\bf u}^+ >(\sum_{j=1}^{i-1} {\bf u}_j)+{\bf u}_i^+ \ \text{ for all } \ i=2,\ldots,l.
\]
When $l=2$, we simply write  ${\bf u}={\bf u}_1+_{ssc} {\bf u}_2$. Note that ${\bf u}={\bf u}_1+_{ssc} {\bf u}_2$ implies that ${\bf u}^+>{\bf u}_1^+$ and ${\bf u}^->{\bf u}_2^-$.
We say that the decomposition is {\bf proper} if all ${\bf u}_1,\ldots,{\bf u}_l$ are nonzero.
We remark that if
${\bf u}=_{ssc}{\bf u}_1+\cdots+{\bf u}_l$   is proper then
 ${\bf u}^+,{\bf u}^+-{\bf u}_1,\ldots,{\bf u}^+-\sum_{i=1}^l {\bf u}_i={\bf u}^-\in \NN^n$ and thus are distinct elements of
 $ \mathcal{F}_{{\bf u}}$.

 We also have the following characterisation of the elements of the universal Markov basis as shown in \cite{HaraThomaVladoiu14}.

\begin{Proposition}\label{universalmarkov} The universal Markov basis ${\MM}(\MA)$ of $\MA$ consists  of all nonzero vectors in $\ML({\MA})$ with no proper strongly  semiconformal decomposition.
\end{Proposition}

In fact, as shown in \cite{HaraThomaVladoiu14}, we have the following relationship between these decompositions
\[ \textrm{ proper conformal } \Rightarrow \textrm { proper strongly semiconformal }  \Rightarrow \textrm { proper semiconformal }   . \]

Let ${\bf u}\in\ML(\MA)$. The {\em fiber} $\mathcal{F}_{{\bf u}}$  is the set $\{{\bf t}\in\NN^n : {{\bf u}}^+-{\bf t}\in \ML(\MA)\}$. We have that $\mathcal{F}_{{\bf u}}$ is  a finite set, since $\ML({\MA})\cap \NN^n=\{\bf 0\}$.

\begin{Proposition}
\label{semiconformal}
Let ${\bf u}\in\ML({\MA})$. There is a bijection between the elements of the fiber $\mathcal{F}_{{\bf u}}$ and the ways that ${\bf u}$ can be written as semiconformal decomposition.
\end{Proposition}

\begin{proof} Let ${\bf t}\in\NN^n$ be in the fiber $\mathcal{F}_{{\bf u}}$. Then $ {{\bf u}}^+-{\bf t}\in \ML(\MA)$  as well as
${\bf t}-{{\bf u}}^- \in \ML(\MA)$, since both ${{\bf u}}^+, {{\bf u}}^-$ belong to $\mathcal{F}_{{\bf u}}$. Set ${\bf v}={{\bf u}}^+-{\bf t}$ and
${\bf w}={\bf t}-{{\bf u}}^-$. Then ${\bf u}={\bf v}+ {\bf w}$ and ${{\bf u}}^+\geq {{\bf v}}^+$ and ${{\bf u}}^-\geq {{\bf w}}^-$, since ${\bf t}\in\NN^n$.
This implies that ${\bf u}={\bf v}+_{sc} {\bf w}$.

For the converse, suppose we have a semiconformal decomposition ${\bf u}={\bf v}+_{sc} {\bf w}$, where  ${\bf u}, {\bf v}, {\bf w} \in \ML(\MA)$. Then ${\bf u}^+\geq {\bf v}^+$ and ${\bf u}^-\geq {\bf w}^-$,
which implies that ${\bf u}^+-{\bf v}={\bf u}^-+ {\bf w}\in \NN^n$. Note that ${\bf u}^+-({\bf u}^+-{\bf v})={\bf v}\in \ML(\MA)$. This implies that
${\bf u}^+-{\bf v}={\bf u}^-+ {\bf w}$ is an  element in the fiber $\mathcal{F}_{{\bf u}}$.
\end{proof}

 Proposition 4.13 in \cite{St} states that certain bases of a toric ideal behave well with respect to elimination.
 Let $\MB \subset \MA$, then for  the Graver bases we have $ \MG(\MB)=  \MG(\MA) \cap \ML(\MB)$.  The corresponding statement is true also for the universal Gr\"obner bases and
 for the circuits.

 However, the corresponding statement in general is not true for the Markov bases or the universal Markov bases i.e.
$$\MM(\MB) \neq \MM(\MA) \cap \ML(\MB) \text{ .}$$ For example, generic toric ideals  \cite{PS} are toric ideals generated by binomials with full support and all elements in a minimal
Markov basis are indispensable, which means that the universal Markov basis is a minimal Markov basis.
However, generic toric ideals are generated by binomials with full support therefore
it follows that $\MM(\MA) \cap \ML(\MB)=\emptyset $
if $\MB$ is a proper subset of $\MA$. This shows that $\MM(\MB) \neq \MM(\MA) \cap \ML(\MB) $ whenever the ideal $I_\MB$ is not zero,
for a generic toric ideal $I_\MA$.
On the contrary, Markov bases of Lawrence liftings behave well with respect to
certain eliminations, which is the content of the next Theorem.

For $A\in\MM_{m\times n}(\NN)$ as above and  $r\ge 2$, the $r$--th {\it Lawrence lifting} of $A$   is denoted by $A^{(r)}$ and is the $(rm+n)\times rn$ matrix

\[
A^{(r)}=\begin{array}{c} \overbrace{\quad\quad\quad \quad \quad\ \ }^{r-\textrm{times}} \\
 \left( \begin{array}{cccc}
\ A\ & 0 &   & 0 \\
0 &\ A\ &  & 0 \\
 & & \ddots &  \\
0 & 0 &  &\ A\ \\
I_n &I_n& \cdots & I_n
\end{array} \right)\end{array},\]
see \cite{SS}. We write $\ML(\MA^{(r)})$ for $\Ker_{\ZZ}(A^{(r)})$ and  identify an element of $\mathcal{L}(\MA^{(r)})$ with  an $r\times n $ matrix:   each row of this matrix corresponds to an element of $\mathcal{L} (\MA)$ and the sum of its rows is zero. The   {\em type} of an element of $\mathcal{L}(\MA^{(r)})$ is the number of  nonzero rows of this matrix. The {\em Markov complexity}, $m(\MA)$, is the largest type of any vector in the universal Markov basis of $A^{(r)}$ as $r$ varies.
According to \cite[Theorem 3.3]{HaraThomaVladoiu15}, since  $\ML({\MA})\cap \NN^n=\{\bf 0\}$ all minimal Markov bases
of $\MA^{(r)}$ have the same complexity for $r \geq 2$.

Let $\MB \subset \MA=\{a_1, a_2, \cdots, a_n\}$ and after renumeration $\MB=\{a_1, a_2, \cdots, a_s\}$. Let ${\bf u}\in \mathcal{L}(\MB^{(r)})$, then we denote by $\sigma ({\bf u})$ an element
of $\mathcal{L}(\MA^{(r)})$ which when is written as an  $r\times n$  matrix the first $s$ columns are the columns of ${\bf u}$ and the last $n-s$ columns are zero columns.
Let ${\bf v}\in \mathcal{L}(\MA^{(r)})$, then we denote by $\pi ({\bf v})$ an  $r\times s$  matrix with columns the first $s$ columns of ${\bf v}$.
In general $\pi ({\bf v})\not \in \mathcal{L}(\MB^{(r)}) $, but if the last $n-s$ columns of ${\bf v}$ are zero then $\pi ({\bf v}) \in \mathcal{L}(\MB^{(r)}) $.

For simplicity, we will denote  $\sigma (\MM(\MB^{(r)}))$ by $\MM(\MB^{(r)})$ and $\sigma (\ML(\MB^{(r)}))$ by $\ML(\MB^{(r)})$.

 \begin{Theorem}
 \label{restriction}
For the universal Markov bases $\MM(\MA^{(r)})$ and $\MM(\MB^{(r)})$ of $\MA^{(r)}$ and $\MB^{(r)}$ respectively, it holds that $\MM(\MB^{(r)}) = \MM(\MA^{(r)}) \cap \ML(\MB^{(r)})$.
\end{Theorem}

 \begin{proof}
We will first show that $\MM(\MB^{(r)}) \subseteq \MM(\MA^{(r)}) \cap \ML(\MB^{(r)})$.
  Let ${\bf u}$ be an element of the universal Markov basis $\MM(\MB^{(r)})$, then  ${\bf u} \in \mathcal{L} (\MB^{(r)})$.

   Suppose that  $\sigma({\bf u}) \notin \MM(\MA^{(r)})$, then by Proposition~\ref{universalmarkov}, there is a proper strongly semiconformal decomposition of $\sigma({\bf u})$
 $$\sigma({\bf u})=_{ssc}{\bf u}_1+\cdots+{\bf u}_l$$
 where each ${\bf u}_i$ is an element of the lattice $\mathcal{L} (\MA^{(r)})$. From the way the element $\sigma({\bf u})$ is defined, the last  $n-s$ columns of the matrix $\sigma({\bf u})$ are zero.
  We claim that all ${\bf u}_j$ also have the last $n-s$ columns equal to zero.
Let us consider one element on the $j$-th column of the last $n-s$ columns of ${\bf u}_1$ that is non-zero.
Since the sum of the entries of each column of the matrix ${\bf u}_1$ are zero, there exists at least one element on the $j$-th column of ${\bf u}_1$ which is positive.
Suppose this element is the element $({\bf u}_1)_{ij}$ which lies on the $i$-th row and $j$-th column. But then ${\bf u}^+>{\bf u}_1^+$ and ${\bf u}_{ij}=0 < ({\bf u}_1)_{ij}$, which is a contradiction.
Therefore, the whole column $j$ would be zero and subsequently each of the $n-s$ last columns would be zero.

We will continue by induction on the number $t$ of elements ${\bf u}_1, \cdots, {\bf u}_{t}$ for which this happens. Suppose that
for some $t$ the last $n-s$ columns of the elements ${\bf u}_1, \cdots, {\bf u}_{t-1}$ are zero.
Let us consider one element on the $j$-th column  of ${\bf u}_t$ that is non-zero, where $s+1\leq j\leq n$.
Since the sum of the entries of each column of the matrix ${\bf u}_t$ are zero, there exists at least one element on the $j$-th column  of ${\bf u}_t$ which is positive.
Suppose this element is the element $({\bf u}_t)_{ij}$ which lies on the $i$-th row and $j$-th column. But then
${\bf u}^+ >(\sum_{j=1}^{t-1} {\bf u}_j)+{\bf u}_t^+ $ and ${\bf u}_{ij}=0 < ({\bf u}_t)_{ij}$, which is a contradiction.

Therefore, all ${\bf u}_j$ have  the last $n-s$ columns equal to zero, which means that $\pi({\bf u}_j)\in \MM(\MB^{(r)})$ for $j=1,\cdots, l$. Thus,
$
 {\bf u}  =_{ssc}\pi ({\bf u}_1)+ \cdots + \pi ({\bf u}_l),
$ which according to Proposition~\ref{universalmarkov} is a contradiction, since ${\bf u}\in \MM(\MB^{(r)})$ and as such should have no proper strongly semiconformal decomposition.

To prove the direction $\MM(\MB^{(r)}) \supset \MM(\MA^{(r)}) \cap \ML(\MB^{(r)})  \text{ ,}$ let ${\bf v} \in \MM(\MA^{(r)}) \cap \ML(\MB^{(r)})$.
If we assume that $\pi({\bf v}) \notin \MM(\MB^{(r)})$, then there exists a proper strongly semiconformal decomposition
$\pi ({\bf v}) =_{ssc} {\bf v}_1 + \cdots +{\bf v}_l$ with each ${\bf v}_i \in \mathcal{L}(\MB^{(r)})$.
But then ${\bf v} =_{ssc} \sigma ({\bf v}_1) + \cdots +\sigma ({\bf v}_l)$ with each $\sigma ({\bf v}_i)\in \mathcal{L}(\MA^{(r)})$. According to Proposition~\ref{universalmarkov}, this is a contradiction since ${\bf v} \in \MM(\MA^{(r)})$.
\end{proof}

As an application, we show that if the Markov complexity $m(\MA)$ of a monomial curve $\MA$ is equal to 2,
then for any subset of three elements $\MB \subset \MA$ the corresponding toric ideal $I_{\MB}$ is complete intersection.

\begin{Corollary}\label{type2} If a monomial curve $\MA= (l_1,l_2, \cdots, l_m)$ in $\mathbb{A}^m$ has Markov complexity 2, then for any $i,j,k$ in $\{1, 2, \cdots, m\}$
the monomial curve $\MB =(l_i, l_j, l_k)$ in $\mathbb{A}^3$ is complete intersection.
\end{Corollary}

\begin{proof} Suppose that there exist $i,j,k$ in $\{1, 2, \cdots, m\}$ such that
the monomial curve $\MB =(l_i, l_j, l_k)$ in $\mathbb{A}^3$ is not complete intersection.

 Then by Theorem 2.6 in \cite{HaraThomaVladoiu14} we know that $m(\MB)=3$.
This means that for any $r$-th Lawrence lifting  $r\geq 3$, $\MB^{(r)}$ has type $3$ elements inside the universal Markov basis
$\MM(\MB^{(r)})$. By Theorem~\ref{restriction}, there is a type 3 element inside $\MM(\MA^{(r)})$ as well.
This means that the Markov complexity is $m(\MA)\geq 3$. A contradiction.
\end{proof}

\begin{Remark} {\rm We note that the converse of Corollary \ref{type2} is not true.
In the next sections, we will give examples of monomial curves  $\MA= (l_1,l_2, \cdots, l_m)$ in $\mathbb{A}^m$ with
arbitrary large Markov complexity, such that for any $i,j,k$ in $\{1, 2, \cdots, m\}$ the monomial curve
$\MB =(l_i, l_j, l_k)$ in $\mathbb{A}^3$ is complete intersection.}
\end{Remark}

\section{The family of monomial curves $\MA_n= \{ 1, n, n^2-n, n^2-1\}$}
\label{Section: MonomialCurve}

In this section we give the guiding example of the paper; a family of monomial curves $\MA_n$ in $\mathbb{A}^4$
which has the special structure that all members in the family are complete intersections,
but also for each one $\MA_n$, any curve in $\mathbb{A}^3$ obtained
by taking any three elements of the curve $\MA_n$ is also complete intersection.
We will also present here some properties governing some semiconformal sums associated to these monomial curves.

Let us consider the example of the monomial curve $\MA_n= \{ 1, n, n^2-n, n^2-1\}$.
 For this curve, there is always the following element of type $n$ in $\mathcal{L}(\MA_n^{(n)})$:

\[
{\bf u}= \left( \begin{array}{cccc}
1 & -1 & -1  & 1 \\
1 & -1 & -1  & 1 \\
& \ddots & &  \\
1 & -1 & -1  & 1 \\
0 & 0 & n+1 & -n \\
2-n &n-2& -3 & 2
\end{array} \right) \text{,} \]
since every row is in $\mathcal{L}(\MA_n)$ and the sum of each column is zero.
Note that the first $n-2$ rows are of the form $(1,-1,-1,1)$, while the last two are
$(0, 0, n+1, -n)$ and
$(n-2, 2-n, -3, 2)$.
The following Lemmas study the ways that two of the above elements of $\mathcal{L}(\MA_n)$
can be written semiconformally under some special conditions. Note that for big $n$, there are thousands of elements in the fibers
of the above elements, which according to Proposition \ref{semiconformal}, means that there are thousand of different ways of writing these elements as semiconformal sums.

\begin{Lemma}
\label{1}
Consider the element $u=(1,-1,-1,1) \in \ML(\MA_n)$. Then $u$ can be written as a semiconformal decomposition
$u=v+_{sc}w$ with the first element of the first term $v_1=1$ in exactly the following two ways
\begin{eqnarray}
(1, -1, -1,1) &=& (1, -1, -1,1) +_{sc} (0,0,0,0)  \\
(1, -1, -1,1) &=& (1, -n, 0,1) +_{sc} (0,n-1,-1,0) \text{.}
\end{eqnarray}
\end{Lemma}

\begin{proof}
Suppose that $u=v+_{sc}w$
for some  vectors $v,w \in\ML(\MA_n)$. Then Proposition~\ref{semiconformal}
implies that  $u^{+} -v= u^{-}+w = (\alpha, \beta, \gamma, \delta)\in \NN^4$.
Therefore, the semiconformal sum $u=v+_{sc} w$ is alternatively written as

$$ \begin{array}{ll}
u& = (u^{+} - (\alpha, \beta, \gamma, \delta) ) +_{sc} ((\alpha, \beta, \gamma, \delta) - u^{-})\\
& = (1-\alpha, -\beta, -\gamma, 1-\delta) ) +_{sc} (\alpha, \beta-1, \gamma-1, \delta)
\end{array}$$

Since the element $(\alpha, \beta, \gamma, \delta) \in \mathcal{F}_{u}$,
we have that $\deg_{A} {\bf x}^{u} = \deg_{A} {\bf x}^{(\alpha, \beta, \gamma, \delta)}$. This implies that
$$
\alpha + \beta n + \gamma (n^2-n) + \delta (n^2-1) = n^2.
$$
 We are interested in establishing what happens when $\alpha = 0$, since we are in the case that $v_1=1$. In this case
 $
 \beta n + \gamma (n^2-n) + \delta (n^2-1) = n^2
$
where $\beta ,\gamma , \delta \in \NN$. Then $\delta $ can take the values of 1 and 0. Suppose that $\delta =1$ then
$
 \beta n + \gamma (n^2-n)  = 1
$ which is a contradiction, since $n\ge 2$ and $n$ divides $1$. Therefore $\delta =0$. Then $ \beta n + \gamma (n^2-n)  = n^2,$ which has only two solutions: $(\beta ,\gamma)=(1,1)$ or $(\beta ,\gamma)=(n,0)$.
Therefore
$(\alpha, \beta, \gamma, \delta)= (0,1,1,0)$ or $(\alpha, \beta, \gamma, \delta)= (0,n,0,0)$ and this gives us only two cases for the semiconformal decomposition
$
u = (1, -1, -1, 1) +_{sc} (0, 0, 0, 0)
$
or
$
u = (1, -n, 0, 1) +_{sc} (0, n-1, -1, 0) \text{.}
$
\end{proof}

\begin{Lemma}
\label{2}

Consider the element $u=(1,-1,-1,1) \in \ML(\MA_n)$. Then $u$ can be written as a semiconformal decomposition
$u=v+_{sc}w$ with the second element of the second term $w_2=-1$ in exactly the following three ways
\begin{eqnarray}
(1, -1, -1,1) &=&(0,0,0,0) +_{sc} (1, -1, -1,1)\\
(1, -1, -1,1) &=&(1-n,0,-1,1) +_{sc} (n,-1,0,0)\\
(1, -1, -1,1) &=&(1-n^2,0,0,1) +_{sc} (n^2,-1,-1,0) \text{.}
\end{eqnarray}
\end{Lemma}

\begin{proof}
Suppose that $u=v+_{sc}w$
for some  vectors $v,w\in\ML(\MA_n)$. Then by Proposition~\ref{semiconformal}
the semiconformal sum $u=v+_{sc} w$, is alternatively written as
\begin{eqnarray}
u &= &(u^{+} - (\alpha, \beta, \gamma, \delta) ) +_{sc} ((\alpha, \beta, \gamma, \delta) - u^{-})\\
 & = & (1-\alpha, -\beta, -\gamma, 1-\delta) ) +_{sc} (\alpha, \beta-1, \gamma-1, \delta)
\end{eqnarray}
where the element $(\alpha, \beta, \gamma, \delta) \in \mathcal{F}_{u}$. We have that $\deg_{A} {\bf x}^{u} = \deg_{A} {\bf x}^{(\alpha, \beta, \gamma, \delta)}$. This means that
$$
\alpha + \beta n + \gamma (n^2-n) + \delta (n^2-1) = n^2.
$$
 We are interested in establishing what happens when $\beta = 0$, since we are in the case that $w_2=-1$. In this case
$
 \alpha + \gamma (n^2-n) + \delta (n^2-1) = n^2
$
where $\alpha ,\gamma , \delta \in \NN^n$. Then $\delta $ can take the values of 1 and 0. Indeed, if $\delta \geq 2$ then
$
 \alpha + \gamma (n^2-n) + \delta (n^2-1) > n^2
$ which is contradiction. If $\delta =1$ then $\alpha =1$ and $\gamma =0$, therefore
$(\alpha, \beta, \gamma, \delta)= (1,0,0,1)$. In the case that $\delta =0$,
we get $ \alpha  + \gamma (n^2-n)  = n^2,$ which has only two solutions, namely
 $(\alpha, \beta, \gamma, \delta)= (n,0,1,0)$ and $(\alpha, \beta, \gamma, \delta)= (n^2,0,0,0)$.

 Therefore, by equation (7) we only have the following three cases for the semiconformal decomposition
$u = (0,0,0,0) +_{sc} (1, -1, -1,1)$ or $u =(1-n,0,-1,1) +_{sc} (n,-1,0,0)$ or $u =(1-n^2,0,0,1) +_{sc} (n^2,-1,-1,0) \text{.}$
\end{proof}

\begin{Lemma}
\label{3}
Consider the element $u=(2-n,n-2,-3,2) \in \ML(\MA_n)$.
If $u$ can be written as a semiconformal decomposition $u=v+_{sc}w$
with the first entries $v_1, w_1$ non-positive and the second entries $v_2, w_2$ non-negative,
then $v={\bf 0}$ or $w={\bf 0}$.
\end{Lemma}

\begin{proof}
Suppose that $u=v+_{sc}w$ for some nonzero vectors $v,w \in\ML(\MA_n)$.
The semiconformal sum $u=v+_{sc} w$, is alternatively written as
\begin{eqnarray}
u& = & (u^{+} - (\alpha, \beta, \gamma, \delta) ) +_{sc} ((\alpha, \beta, \gamma, \delta) - u^{-})\\
& = & (-\alpha, (n-2) -\beta, -\gamma, 2-\delta) ) +_{sc} (\alpha - (n-2), \beta, \gamma-3, \delta)
\end{eqnarray}
 Since $(\alpha, \beta, \gamma, \delta)$ belongs to the fiber $\mathcal{F}_{u}$, we have that $\deg_{A} {\bf x}^{u} = \deg_{A} {\bf x}^{(\alpha, \beta, \gamma, \delta)}$. This means that
\begin{eqnarray}{c}
\label{DegreeOfFibreLastRow}
\alpha + \beta n + \gamma (n^2-n) + \delta (n^2-1) = 3n^2 - 2n -2
\end{eqnarray}
which also gives $\alpha- \delta\equiv -2 \mod n$. The initial conditions about the entries $v_1, w_1$ and the entries $v_2, w_2$ imply that $0\leq \alpha \leq n-2$ and $0\leq \beta \leq n-2$.

 Noting that $0 \leq \delta \leq 2$, we distinguish three cases for the value of $\delta$.
 In the case that $\delta = 2$, we have that $\alpha\equiv 0 \mod n$ which together with $0\leq \alpha \leq n-2$ imply that $\alpha =0$. Then equation (\ref{DegreeOfFibreLastRow}) gives $ \beta n + \gamma (n^2-n)  = n^2 - 2n$, which in turn implies that $\gamma =0$ and $\beta =n-2$. Therefore
$(\alpha, \beta, \gamma, \delta)= (0,n-2,0,2)$ obtaining the semiconformal decomposition
$
u = (0, 0, 0, 0) +_{sc} (2-n, n-2, -3, 2) \text{.}
$

Now if $\delta = 1$, we get that $\alpha\equiv -1 \mod n$ and together with $0\leq \alpha \leq n-2$ gives a contradiction. \\
Finally, if $\delta = 0$, then $\alpha\equiv -2 \mod n$ which together with $0\leq \alpha \leq n-2$ imply that $\alpha =n-2$. Then equation (\ref{DegreeOfFibreLastRow}) becomes
 $ n-2+ \beta n + \gamma (n^2-n)  = 3n^2 - 2n -2$ which in turn gives  $ \beta + \gamma (n-1)  = 3n - 3$. This means that $\beta$ is a multiple of $n-1$ and since $0\leq \beta \leq n-2$ the only option is for $\beta =0$ and $\gamma =3$. Therefore
$(\alpha, \beta, \gamma, \delta)= (n-2,0,3,0)$ gives us the semiconformal decomposition
$
u =  (2-n, n-2, -3, 2) +_{sc} (0, 0, 0, 0) \text{.}
$
\end{proof}

\section{Markov complexity of monomial curves}
\label{Section: MainResult}

In this section, we prove the main result of this paper regarding the unboundedness of the Markov complexity
of monomial curves in $\mathbb{A}^m, m\geq 4$. We use the properties of semiconformal decompositions for the special
monomial curve $\MA_n= \{ 1, n, n^2-n, n^2-1\}$ shown in Section~\ref{Section: MonomialCurve}, as well as
Theorem~\ref{restriction} regarding the good behaviour of  Markov bases of higher Lawrence liftings
with respect to elimination.

\begin{Theorem}
\label{large} Monomial curves in $\mathbb{A}^4$ may have arbitrary large Markov complexity.
\end{Theorem}

\begin{proof}
We will show that the type $n$ element
\[
{\bf u} = \left( \begin{array}{cccc}
1 & -1 & -1  & 1 \\
1 & -1 & -1  & 1 \\
 & & \ddots &  \\
1 & -1 & -1  & 1 \\
0 & 0 & n+1 & -n \\
2-n & n- 2& -3 & 2
\end{array} \right)
\]
 belongs to every Markov basis of $\MA_n^{(n)}$.
Which means we wish to show that the element ${\bf u}$ is indispensable, namely that it belongs to $S(\MA^{(n)})$,
the intersection of all the minimal Markov bases. Let us assume on the contrary that the element ${\bf u}$ is not indispensable.
Proposition 1.1 in \cite{HaraThomaVladoiu14}, implies that ${\bf u}$ admits a proper semiconformal
decomposition ${\bf u}={\bf v}+_{sc} {\bf w}$, where ${\bf u},{\bf v},{\bf w} \in \mathcal{L}(\MA_n^{(n)})$ such that
$$
{\bf v}_{ij}>0 \Rightarrow {\bf w}_{ij} \geq 0 \text{ and }
{\bf w}_{ij}<0 \Rightarrow {\bf v}_{ij} \leq 0 \text{, }
$$
for any $1 \leq i \leq n, 1 \leq j \leq 4$. In terms of signs, for each row of the vector ${\bf u}$ we have
\begin{eqnarray*}
(1,-1,-1,1) & = & ( * , \ominus , \ominus , * ) +_{sc} ( \oplus ,* , * , \oplus) \\
(0, 0, n+1, -n) & = & (\ominus,\ominus,*,\ominus) +_{sc} (\oplus,\oplus,\oplus,*)\\
(n-2, 2-n, -3, 2) & = & (\ominus,*,\ominus,*) +_{sc} (*,\oplus,*,\oplus) \text{.}
\end{eqnarray*}

The symbol $\ominus$ means that the corresponding integer is non positive, the symbol $\oplus$ non negative and the symbol $*$ means that it can take any value.

  Let ${\bf u}={\bf v}+_{sc} {\bf w}$ be a semiconformal decomposition of ${\bf u}$, then the sign pattern of the elements ${\bf v}, {\bf w}$ is:
 \[
{\bf u} = \left( \begin{array}{cccc}
1 & -1 & -1  & 1 \\
1 & -1 & -1  & 1 \\
 & & \ddots &  \\
1 & -1 & -1  & 1 \\
0 & 0 & n+1 & -n \\
2-n &n-2& -3 & 2
\end{array} \right)
= \left( \begin{array}{cccc}
* & \ominus & \ominus  & * \\
* & \ominus & \ominus & * \\
 & & \ddots &  \\
* & \ominus & \ominus  & * \\
\ominus & \ominus & * & \ominus \\
\z \ominus & \z * & \ominus & *
\end{array} \right)
+_{sc}
\left( \begin{array}{cccc}
\oplus & * & *  & \oplus \\
\oplus & * & *  & \oplus \\
 & & \ddots &  \\
\oplus & * & *  & \oplus \\
\oplus & \oplus & \oplus & * \\
\z * & \z \oplus & * & \oplus
\end{array} \right)
\]

Considering that the sum of every column should be zero, we conclude that the last element of the second column of ${\bf v}$, ${\bf v}_{n,2}$, is
non-negative and the last element of the first column of ${\bf w}$, ${\bf w}_{n,1}$, is non-positive. This means that in the $n^{th}$ row the elements  highlighted in grey above are;
$ {\bf v}_{n,1},  {\bf w}_{n,1}$ which are non-positive and the elements $ {\bf v}_{n,2},  {\bf w}_{n,2}$ which are non-negative. From Lemma~\ref{3}, we distinguish two cases for the last row: first case that the last row of ${\bf w}$ is zero or second case that the last row of ${\bf v}$ is zero.

In the first case the decomposition of ${\bf u}$ becomes
\[
 \left( \begin{array}{cccc}
1 & -1 & -1  & 1 \\
1 & -1 & -1  & 1 \\
 & & \ddots &  \\
1 & -1 & -1  & 1 \\
0 & 0 & n+1 & -n \\
2-n &n-2& -3 & 2
\end{array} \right)
= \left( \begin{array}{cccc}
* & \ominus & \ominus  & * \\
* & \ominus & \ominus & * \\
 & \ddots & &  \\
* & \ominus & \ominus  & * \\
\ominus & \ominus & * & \ominus \\
2-n &n-2& -3 & 2
\end{array} \right)
+_{sc}
\left( \begin{array}{cccc}
\z \oplus & * & *  & \oplus \\
\z \oplus & * & *  & \oplus \\
\z & \ddots & &  \\
\z\oplus & * & *  & \oplus \\
\z \oplus & \oplus & \oplus & * \\
\z 0 & 0 & 0 & 0
\end{array} \right)
\]

The first column of ${\bf w}$, highlighted in gray above, is non negative and adds to zero, thus, all the column is zero. So,
\[
\left( \begin{array}{cccc}
1 & -1 & -1  & 1 \\
1 & -1 & -1  & 1 \\
 & \ddots & &  \\
1 & -1 & -1  & 1 \\
0 & 0 & n+1 & -n \\
2-n &n-2& -3 & 2
\end{array} \right)
= \left( \begin{array}{cccc}
1 & \z \ominus & \ominus  & * \\
1 & \z \ominus & \ominus & * \\
 & \z & \ddots &  \\
1 & \z \ominus & \ominus  & * \\
0 & \z \ominus & * & \ominus \\
2-n & \z n-2& -3 & 2
\end{array} \right)
+_{sc}
\left( \begin{array}{cccc}
0 & * & *  & \oplus \\
0 & * & *  & \oplus \\
 & & \ddots &  \\
0 & * & *  & \oplus \\
0 & \oplus & \oplus & * \\
0 & 0 & 0 & 0
\end{array} \right) \text{.}
\]

By Lemma~\ref{1}, we see that if the first element of the first vector in the semiconformal decomposition of $(1, -1, -1,1)$ is $1$, then there are exactly two ways of decomposing
semiconformally the element $(1, -1, -1,1)$; namely
$$(1, -1, -1,1) +_{sc} (0,0,0,0) \text{ or } (1, -n, 0,1) +_{sc} (0,n-1,-1,0) \text{.}$$

This means that each one of the first $n-2$ rows of ${\bf v}$ should be either $(1, -1, -1,1)$ or $(1, -n, 0,1) $. By looking at the second
column of ${\bf v}$, highlighted in gray above, we see that the first $n-2$ elements are either $-1$ or $-n$ , the $(n-1)^{th}$ element is non-positive and the last is $n-2$. Since they add to zero, this
forces all of the first $n-2$ elements to be $-1$ and the $(n-1)^{th}$ element to be zero. So $-1$ in the second position means that we are in the first decomposition
$(1, -1, -1,1)= (1, -1, -1,1) +_{sc} (0,0,0,0)$. Therefore,

\[
 \left( \begin{array}{cccc}
1 & -1 & -1  & 1 \\
1 & -1 & -1  & 1 \\
 & & \ddots &  \\
1 & -1 & -1  & 1 \\
0 & 0 & n+1 & -n \\
2-n &n-2& -3 & 2
\end{array} \right)
= \left( \begin{array}{cccc}
1 & -1 & -1  & 1 \\
1 & -1 & -1  & 1 \\
 & & \ddots &  \\
1 & -1 & -1  & 1 \\
0 & 0 & * & \ominus \\
2-n &n-2& -3 & 2
\end{array} \right)
+_{sc}
\left( \begin{array}{cccc}
0 & 0 & 0  & 0 \\
0 & 0 & 0  & 0 \\
 & & \ddots &  \\
0 & 0 & 0  & 0 \\
0 & 0 & \z \oplus & \z * \\
0 & 0 & 0 & 0
\end{array} \right) \text{.}
\]

Looking at ${\bf w}$, in particular the entries highlighted in gray above, and considering that each column adds to zero, we have that
\[
\left( \begin{array}{cccc}
1 & -1 & -1  & 1 \\
1 & -1 & -1  & 1 \\
 & & \ddots &  \\
1 & -1 & -1  & 1 \\
0 & 0 & n+1 & -n \\
2-n &n-2& -3 & 2
\end{array} \right)
= \left( \begin{array}{cccc}
1 & -1 & -1  & 1 \\
1 & -1 & -1  & 1 \\
 & & \ddots &  \\
1 & -1 & -1  & 1 \\
0 & 0 & n+1 & -n \\
2-n &n-2& -3 & 2
\end{array} \right)
+_{sc}
\left( \begin{array}{cccc}
0 & 0 & 0  & 0 \\
0 & 0 & 0  & 0 \\
 & & \ddots &  \\
0 & 0 & 0  & 0 \\
0 & 0 & 0 & 0 \\
0 & 0 & 0 & 0
\end{array} \right) \text{.}
\]

In the second case, the decomposition of ${\bf u}$ becomes:

\[
 \left( \begin{array}{cccc}
1 & -1 & -1  & 1 \\
1 & -1 & -1  & 1 \\
 & & \ddots &  \\
1 & -1 & -1  & 1 \\
0 & 0 & n+1 & -n \\
2-n &n-2& -3 & 2
\end{array} \right)
= \left( \begin{array}{cccc}
* & \z \ominus & \ominus  & * \\
* & \z \ominus & \ominus & * \\
 & \z & \ddots &  \\
* & \z \ominus & \ominus  & * \\
\ominus & \z \ominus & * & \ominus \\
0 & \z 0 & 0 & 0
\end{array} \right)
+_{sc}
\left( \begin{array}{cccc}
\oplus & * & *  & \oplus \\
\oplus & * & *  & \oplus \\
 & & \ddots &  \\
\oplus & * & *  & \oplus \\
\oplus & \oplus & \oplus & * \\
2-n & n-2 & -3 & 2
\end{array} \right) \text{.}
\]

The second column of ${\bf v}$, highlighted in gray above, is non positive and adds to zero, thus, all the column is zero. Then

\[
 \left( \begin{array}{cccc}
1 & -1 & -1  & 1 \\
1 & -1 & -1  & 1 \\
 & & \ddots &  \\
1 & -1 & -1  & 1 \\
0 & 0 & n+1 & -n \\
2-n &n-2& -3 & 2
\end{array} \right)
= \left( \begin{array}{cccc}
* & 0 & \ominus  & * \\
* & 0 & \ominus & * \\
 & & \ddots &  \\
* & 0 & \ominus  & * \\
\ominus & 0 & * & \ominus \\
0 & 0 & 0 & 0
\end{array} \right)
+_{sc}
\left( \begin{array}{cccc}
\z \oplus & -1 & *  & \oplus \\
\z \oplus & -1 & *  & \oplus \\
\z & & \ddots &  \\
\z \oplus & -1 & *  & \oplus \\
\z \oplus & 0 & \oplus & * \\
\z 2-n & n-2 & -3 & 2
\end{array} \right) \text{.}
\]

By Lemma~\ref{2}, we have that if the second element of the second vector in the semiconformal decomposition of $(1, -1, -1,1)$ is $-1$, then there are exactly three ways of decomposing
semiconformally the $(1, -1, -1,1)$; namely
$(0,0,0,0) +_{sc} (1, -1, -1,1)$ or $(1-n,0,-1,1) +_{sc} (n,-1,0,0)$ or $(1-n^2,0,0,1) +_{sc} (n^2,-1,-1,0) \text{.}$

This means that each one of the first $n-2$ rows of ${\bf w}$ should be either $(1, -1, -1,1)$ or $(n,-1,0,0)$ or $(n^2,-1,-1,0)$. By looking at the first
column of ${\bf w}$, highlighted in gray above, we see that the first $n-2$ elements are either $1$ or $n$ or $n^2$, the $(n-1)^{th}$ element is non-negative and the last is $2-n$. They all add to zero, so that
forces all of the first $n-2$ elements to be $1$ and the $(n-1)^{th}$ to be zero. Having $1$ in the first position means that we are in the first decomposition
$(1, -1, -1,1)= (0,0,0,0) +_{sc} (1, -1, -1,1)$. Therefore,

\[
 \left( \begin{array}{cccc}
1 & -1 & -1  & 1 \\
1 & -1 & -1  & 1 \\
 & & \ddots &  \\
1 & -1 & -1  & 1 \\
0 & 0 & n+1 & -n \\
2-n &n-2& -3 & 2
\end{array} \right)
= \left( \begin{array}{cccc}
0 & 0 & 0  & 0 \\
0 & 0 & 0 & 0 \\
 & & \ddots &  \\
0 & 0 & 0  & 0 \\
0 & 0 & \z * & \z \ominus \\
0 & 0 & 0 & 0
\end{array} \right)
+_{sc}
\left( \begin{array}{cccc}
1 & -1 & -1  & 1 \\
1 & -1 & -1  & 1 \\
 & & \ddots &  \\
1 & -1 & -1  & 1 \\
0 & 0 & \oplus & * \\
2-n & n-2 & -3 & 2
\end{array} \right) \text{.}
\]

Looking at ${\bf v}$, particularly the entries highlighted in gray above, and considering that each column adds to zero, we have that

\[
 \left( \begin{array}{cccc}
1 & -1 & -1  & 1 \\
1 & -1 & -1  & 1 \\
 & & \ddots &  \\
1 & -1 & -1  & 1 \\
0 & 0 & n+1 & -n \\
2-n &n-2& -3 & 2
\end{array} \right)
= \left( \begin{array}{cccc}
0 & 0 & 0  & 0 \\
0 & 0 & 0 & 0 \\
 & & \ddots &  \\
0 & 0 & 0  & 0 \\
0 & 0 & 0 & 0 \\
0 & 0 & 0 & 0
\end{array} \right)
+_{sc}
\left( \begin{array}{cccc}
1 & -1 & -1  & 1 \\
1 & -1 & -1  & 1 \\
 & & \ddots &  \\
1 & -1 & -1  & 1 \\
0 & 0 & n+1 & -n \\
2-n & n-2 & -3 & 2
\end{array} \right) \text{.}
\]

 Therefore the decomposition ${\bf u}={\bf v}+_{sc} {\bf w}$ can never been proper.
Thus, we conclude that ${\bf u}$ is indispensable, therefore it belongs to all Markov bases. \end{proof}

\begin{Remark} {\rm
We do not claim that the Markov complexity $m(\MA_n)$ of $\MA_n$ is $n$, but at least $n$. Indeed consider the example of the monomial curve
$\MA_5 = (1,5,20,24) \in \mathbb{A}^4$.
Then using the computational  package 4ti2 (see \cite{4ti2}),
we compute a Markov basis of $\MA_n^{(r)}$ for $r \leq 6$.
Table 1 includes the number of elements of the Markov basis of $\MA_n^{(r)}$ as well as the
largest type of any vector in the universal Markov basis of $\MA_n^{(r)}$ for each $r=1, \cdots 6$.

\begin{table}[h]
\label{ComplexityExample}
\begin{center}
 \begin{tabular}{||c | c | c||}
 \hline
 $r$th Lawrence lifting & \# elements of Markov basis & Type \\ [0.5ex]
 \hline\hline
 2 & 46 & 2 \\
 \hline
 3 & 174 & 3  \\
 \hline
 4 & 528 & 4  \\
 \hline
 5 & 1520 & 5  \\
 \hline
 6 & 4110 & 6  \\
 \hline
 7 & 10206 & 6  \\ [1ex]
 \hline
\end{tabular}
\caption {The monomial curve $\MA_5 = (1,5,20,24) \text{ in } \mathbb{A}^4$}
\end{center}
\end{table}

Therefore, this implies that the Markov complexity is at least 6.
In fact, the elements of type 6 in the sixth Lawrence lifting are

\[
 \left( \begin{array}{cccc}
0 & 0 & -6  & 5 \\
-2 & 2 & -4  & 3 \\
-2 & 2 & -4  & 3 \\
-2 & 2 & -4  & 3 \\
3 & -3 & 3  & -2 \\
3 & -3 & 3  & -2
\end{array} \right),
 \left( \begin{array}{cccc}
0 & 0 & -6  & 5 \\
0 & 0 & -6  & 5 \\
-1 & 1 & -5  & 4 \\
-1 & 1 & -5  & 4 \\
-1 & 1 & -5  & 4 \\
3 & -3 & 3  & -2
\end{array} \right)
\]
as well as all elements coming from permutation of the rows of the above matrices.
}
\end{Remark}

Note that the monomial curve $\MA_n= \{ 1, n, n^2-n, n^2-1\}$ is complete intersection, therefore from Theorem \ref{large} it
directly follows:
\begin{Corollary}
\label{ci} Complete intersection monomial curves in $\mathbb{A}^4$ may have arbitrary large Markov complexity.
\end{Corollary}
\begin{Corollary}
\label{n} Monomial curves in $\mathbb{A}^m$, $m\ge 4$, may have arbitrary large Markov complexity.
\end{Corollary}

 \begin{proof} The proof for general $m\geq 5$ follows from Theorem~\ref{restriction} and Theorem~\ref{large}.

 Suppose there is a $d\in \NN$ such that $m(\MA)\leq d$ for all monomial curves $\MA$ in $\mathbb{A}^m$.
 If we consider the monomial curve $\MA = (1, n, n^2-n, n^2-1, a_{m-4}, \cdots, a_{m})$ in $\mathbb{A}^m$, where
  $n \geq d+1$ and $a_{m-4}, \cdots, a_{m}$ are any natural numbers,
 then this means that the largest type of any vector in the universal Markov basis of $\MA^{(r)}$ as $r$ varies
 would be at most $d$.

 We know by Theorem~\ref{large}, that $m(\MB)\geq n$ for $\MB = \{1, n, n^2-n, n^2-1\} \subset \MA$.
 This means that for any $r$-th Lawrence lifting  $r\geq n$, $\MB^{(r)}$ has an element of type at least
 $n$ inside the universal Markov basis
$\MM(\MB^{(r)})$. By Theorem~\ref{restriction}, there is an element of type at least $n$ inside $\MM(\MA^{(r)})$ as well.
This means that the Markov complexity is $m(\MA)\geq n$. Since $n \geq d+1$, we immediately reach a contradiction.
\end{proof}










\end{document}